\documentclass{amsart}
\usepackage{graphicx}
\usepackage[all]{xy}
\usepackage{amscd}
\usepackage{amssymb}
\usepackage{amsmath}
\usepackage{amsthm}
\usepackage{amsfonts}
\usepackage{graphics}
\usepackage{epsfig,psfrag}
\usepackage[english]{babel}
\usepackage{latexsym}
\usepackage{mathrsfs}

\newtheorem{theorem}{Theorem}[section]

\newtheorem{lemma}[theorem]{Lemma}

\newtheorem{proposition}[theorem]{Proposition}

\theoremstyle{definition}
\newtheorem{example}[theorem]{Example}
\newtheorem{definition}[theorem]{Definition}

\theoremstyle{remark}
\newtheorem{remark}[theorem]{Remark}

\begin{document}

\title[Maximum orders of cyclic and abelian extendable actions]
{Maximum orders of cyclic and abelian extendable actions on surfaces}

\author{Chao Wang}
\address{School of Mathematical Sciences, University of
Science and Technology of China, Hefei 230026, CHINA}
\email{chao\_{}wang\_{}1987@126.com}

\author{Yimu Zhang}
\address{Mathematics School, Jilin University,
Changchun 130012, CHINA}
\email{zym534685421@126.com}



\thanks{The authors were supported by National Natural Science Foundation of China (11371034). The authors would like to thank Prof. Shicheng Wang and the referee for their helpful advices.}

\begin{abstract}
A faithful action of a group $G$ on the genus $g>1$ orientable closed surface $\Sigma_g$ is extendable (over the three dimensional sphere $S^3$), with respect to an embedding $e: \Sigma_g\hookrightarrow S^3$, if $G$ can act on $S^3$ such that $h\circ e=e\circ h$ for any $h \in G$. We show that the maximum order of extendable cyclic group actions on $\Sigma_g$ is $4g+4$ when $g$ is even, and is $4g-4$ when $g$ is odd; the maximum order of extendable abelian group actions on $\Sigma_g$ is $4g+4$. We also give the maximum orders of cyclic and abelian group actions on handlebodies.
\end{abstract}

\date{}
\maketitle

\section{Introduction}

Let $\Sigma_g$ be the genus $g>1$ orientable closed surface, $S^3$ be the three dimensional sphere, and $G$ be a finite group.

\begin{definition}\label{Def of extendable action}
A faithful $G$-action on $\Sigma_g$ is extendable, with respect to an embedding $e: \Sigma_g\hookrightarrow S^3$, if $G$ can act on $S^3$ such that $h\circ e=e\circ h$ for any $h \in G$. We also say that such an action on $\Sigma_g$ is extendable over $S^3$.
\end{definition}

In this paper we consider the following question in the smooth category (namely all manifolds and maps are smooth): what is the maximum order of extendable cyclic (abelian) group actions on $\Sigma_g$ ? And following is our main result.





\begin{theorem}\label{Thm of maximum of CandA action}
(1). The maximum order of extendable cyclic group actions on $\Sigma_g$ is $4g+4$ when $g$ is even, and is $4g-4$ when $g$ is odd.

(2). The maximum order of extendable abelian group actions on $\Sigma_g$ is $4g+4$.

\end{theorem}

Let $V_g$ be the genus $g>1$ orientable handlebody. The question below is closely related to the above one: what is the maximum order of cyclic (abelian) group actions on $V_g$ ? And we have the following answer to it.


\begin{theorem}\label{Thm of CandA on handlebody}
(1). The maximum order of cyclic group actions on $V_g$ is
$$Max\{2[m, n]\mid g=[m, n]-\frac{m+n}{(m, n)}+1, m, n\in \mathbb{Z}_+, 2\nmid m, 2\nmid n\}$$
when $g$ is even, and is
$$Max\{2kn\mid g=kn-k+1, k, n\in \mathbb{Z}_+, 2\nmid k, 2\nmid n\}$$
when $g$ is odd.

(2). The maximum order of abelian group actions on $V_g$ is $4g+4$ when $g\neq 5$, and is $32$ when $g=5$. It is equal to the maximum order of abelian group actions on $\Sigma_g$.
\end{theorem}

Here $\mathbb{Z}_+$ is the set of positive integers, $[m, n]$ is the lowest common multiple of $m$ and $n$, and $(m, n)$ is the greatest common divisor of $m$ and $n$.

If $g$ is even, then let $m=n=g+1$. And if $g$ is odd, then let $k=1, n=g$. Then by Theorem \ref{Thm of CandA on handlebody}, it is easy to see that the maximum order of cyclic group actions on $V_g$ is at least $2g+2$ when $g$ is even, and is at least $2g$ when $g$ is odd.


\begin{remark}\label{Rem of classical result}
(1). The extendable action in Definition \ref{Def of extendable action} was firstly defined in \cite{WWZZ1}. And following notations were used in \cite{WWZZ1}.

$C_g$ ($A_g$): the maximum order of orientation-preserving cyclic (abelian) group actions on $\Sigma_g$.

$CH_g$ ($AH_g$): the maximum order of orientation-preserving cyclic (abelian) group actions on $V_g$.

$CE_g$ ($AE_g$): the maximum order of extendable cyclic (abelian) group actions on $\Sigma_g$, which preserve both the orientations of $\Sigma_g$ and $S^3$.

All these numbers have been determined, see \cite{Ha, St, Wa} for $C_g$; \cite{Ma} for $A_g$; \cite{MMZ} for $CH_g$; \cite{MMZ, RZ1, WWZZ1} for $AH_g$; \cite{WWZZ1} for $CE_g$ and $AE_g$. We summarize them in Table \ref{tab of maxorder of OPcase}:
\begin{table}[h]
\caption{Maximum orders in orientation-preserving case}\label{tab of maxorder of OPcase}
  \begin{tabular}{|l|l|l|l|}
  \hline $C_g$ & $4g+2$ & $A_g$ & $4g+4$\\
  \hline $CH_g$ & $2g+2$ ($g$ even), $2g$ ($g$ odd) & $AH_g$ & $2g+2$ ($g\neq5$), $16$ ($g=5$)\\
  \hline $CE_g$ & $2g+2$ ($g$ even), $2g-2$ ($g$ odd) & $AE_g$ & $2g+2$\\\hline
  \end{tabular}
\end{table}

(2). The maximum order of orientation-reversing periodic maps on $\Sigma_g$ is $4g+4$ when $g$ is even, and is $4g-4$ when $g$ is odd. Hence the maximum order of cyclic group actions on $\Sigma_g$ is $4g+4$ when $g$ is even, and is $4g+2$ when $g$ is odd, see \cite{Wa}.
\end{remark}

By comparing Theorem \ref{Thm of maximum of CandA action} and \ref{Thm of CandA on handlebody} with Remark \ref{Rem of classical result}, we see that in most cases if the maximum order is achieved, then the group will contain an element reversing the orientation of either $\Sigma_g$, $V_g$ or $S^3$.

\begin{definition}\label{Def of MO of ORaction}
For each given $g>1$, define

$C_g^-$: the maximum order of cyclic group actions on $\Sigma_g$, which contains an element reversing the orientation of $\Sigma_g$.

$A_g^-$: the maximum order of abelian group actions on $\Sigma_g$, which contains an element reversing the orientation of $\Sigma_g$.

$CH_g^-$: the maximum order of cyclic group actions on $V_g$, which contains an element reversing the orientation of $V_g$.

$AH_g^-$: the maximum order of abelian group actions on $V_g$, which contains an element reversing the orientation of $V_g$.
\end{definition}

\begin{definition}\label{Def of five type}
Define five types of extendable $G$-actions on $\Sigma_g$ as following:

$(+,+)$: $G$ preserves both the orientations of $\Sigma_g$ and $S^3$.

$(+,-)$: $G$ preserves the orientation of $\Sigma_g$, and there exists $h_1\in G$ such that $h_1$ reverses the orientation of $S^3$.

$(-,+)$: $G$ preserves the orientation of $S^3$, and there exists $h_2\in G$ such that $h_2$ reverses the orientation of $\Sigma_g$.

$(-,-)$: for any $h\in G$, either $h$ preserves both the orientations of $\Sigma_g$ and $S^3$, or $h$ reverses both the orientations of $\Sigma_g$ and $S^3$.

$(Mix)$: there exists $h_1 \in G$ such that $h_1$ preserves the orientation of $\Sigma_g$ and reverses the orientation of $S^3$, and there exists $h_2\in G$ such that $h_2$ reverses the orientation of $\Sigma_g$ and preserves the orientation of $S^3$.
\end{definition}

\begin{definition}\label{Def of MO of type}
For a given type $\mathcal{T}$ in Definition \ref{Def of five type} and a given $g>1$, define

$CE_g\mathcal{T}$: the maximum order of type $\mathcal{T}$ extendable cyclic group actions on $\Sigma_g$, if such an action exists.

$AE_g\mathcal{T}$: the maximum order of type $\mathcal{T}$ extendable abelian group actions on $\Sigma_g$, if such an action exists.
\end{definition}

Clearly we have $CE_g(+,+)=CE_g$ and $AE_g(+,+)=AE_g$. For other types we have the following result. By Remark \ref{Rem of classical result}, it will give us Theorem \ref{Thm of maximum of CandA action}.


\begin{proposition}\label{Pro of maxorder of five type}
For each type $\mathcal{T}\neq (+,+)$ and genus $g>1$, $CE_g\mathcal{T}$ and $AE_g\mathcal{T}$ are given in Table \ref{tab of maxorder of fivetype}. The notation ``---'' means that the type $\mathcal{T}$ extendable action does not exist.
\begin{table}[h]
\caption{Maximum orders of type $\mathcal{T}$ extendable actions}\label{tab of maxorder of fivetype}
  \begin{tabular}{|c|l|l|}
  \hline $\mathcal{T}$ & $CE_g\mathcal{T}$ & $AE_g\mathcal{T}$\\
  \hline $(+,-)$ & $2g+2$ & $4g+4$\\
  \hline $(-,+)$ & $4g+4$ ($g$ even), $4g-4$ ($g$ odd) & $4g+4$\\
  \hline $(-,-)$ & $2g+2$ ($g$ even), $2g$ ($g$ odd) & $4g+4$\\
  \hline $(Mix)$ & --- & $2g+4$ ($g$ even), --- ($g$ odd)\\\hline
  \end{tabular}
\end{table}
\end{proposition}

We see that $CE_g(-,+)=C_g^-$ and $AE_g(+,-)=A_g$, see Remark \ref{Rem of classical result}. Namely some kinds of maximum symmetries on surfaces are extendable over $S^3$.


We also have the following result. By Remark \ref{Rem of classical result}, it will give us Theorem \ref{Thm of CandA on handlebody}.

\begin{proposition}\label{Pro of ORhandlebody}
For each genus $g>1$, $CH_g^-$ and $AH_g^-$ are given in Table \ref{tab of maxorder of ORhandlebody}, here $k, m, n\in \mathbb{Z}_+$. And $A_g^-$ is equal to $AH_g^-$.
\begin{table}[h]
\caption{$CH_g^-$ and $AH_g^-$}\label{tab of maxorder of ORhandlebody}
  \begin{tabular}{|l|l|}
  \hline $CH_g^-$ ($g$ even) & $Max\{2[m, n]\mid g=[m, n]-(m+n)/(m, n)+1, 2\nmid m, 2\nmid n\}$\\
  \hline $CH_g^-$ ($g$ odd) & $Max\{2kn\mid g=kn-k+1, 2\nmid k, 2\nmid n\}$\\
  \hline $AH_g^-$ & $4g+4$ ($g\neq 5$), $32$ ($g=5$)\\\hline
  \end{tabular}
\end{table}
\end{proposition}

In \S2, we will give some examples which will give the lower bounds of maximum orders we have defined. In \S3, we will give some known facts about group actions on manifolds, which will be used later. In \S4, we will prove Proposition \ref{Pro of ORhandlebody}. Some results in \S4 will be used in \S5, where we will prove Proposition \ref{Pro of maxorder of five type}.


\section{Examples}\label{Sec of examples}

In this section we will give three examples about extendable actions and one example of non-extendable action. The surfaces we constructed below may be non-smooth, but using orbifold theory discussed in \S3, they can be replaced by smooth ones. The following lemma is easy to verified.




\begin{lemma}\label{Lem of change sides}
Let $\Sigma_g$ be a closed subsurface in $S^3$. Suppose $h$ is an automorphism of $S^3$ preserving $\Sigma_g$. If some of the following conditions (a), (b) and (c) happen, then exactly two of them happen.

(a). $h$ reveres the orientation of $S^3$.

(b). $h$ reveres the orientation of $\Sigma_g$.

(c). $h$ changes the two sides of $\Sigma_g$.
\end{lemma}

\begin{example}[Cage]\label{Ex of cage}
Given $g>1$, we will construct some cyclic and abelian group actions on a Heegaard splitting of $S^3$.

Let $S^3$ be the unit sphere in $\mathbb{C}^2$:
$$S^3=\{(z_1, z_2)\in\mathbb{C}^2\mid |z_1|^2+|z_2|^2=1\}.$$
Let
\begin{align*}
a_m&=(e^{\frac{m\pi}{2}i}, 0), m=0, 1, 2, 3,\\
b_n&=(0, e^{\frac{n\pi}{g+1}i}), n=0, 1, \cdots, 2g+1.
\end{align*}
Connect each $a_{2l}$ to each $b_{2k}$ with the shortest geodesic in $S^3$, and connect each $a_{2l+1}$ to each $b_{2k+1}$ with the shortest geodesic in $S^3$, here $l=0, 1$ and $k=0, 1, \cdots, g$. Then we get two graphs $\Gamma_g^c, {\Gamma_g^c}'\in S^3$. $\Gamma_g^c$ contains $a_{2l}$, $b_{2k}$, and ${\Gamma_g^c}'$ contains $a_{2l+1}$, $b_{2k+1}$. Each graph has $g+3$ vertices and $2g+2$ edges. They are in the dual positions as in Figure \ref{fig of cages} ($g=2$), here graphs have been projected into the three dimensional Euclidean space $E^3$ from $S^3-\{(-1, 0)\}$. The vertex $a_2$ is at the infinity.

\begin{figure}[h]
\centerline{\scalebox{0.48}{\includegraphics{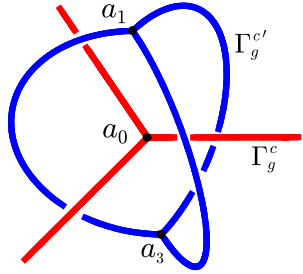}}}
\caption{$\Gamma_g^c$ and ${\Gamma_g^c}'$ in $S^3$}\label{fig of cages}
\end{figure}

Let $T$ be the following torus in $S^3$:
$$T=\{(z_1, z_2)\in S^3\mid |z_1|=|z_2|=\frac{\sqrt{2}}{2}\}.$$
It splits $S^3$ into two solid tori $V_1$ and $V_2$:
\begin{align*}
V_1&=\{(z_1, z_2)\in S^3\mid |z_1|\geq \frac{\sqrt{2}}{2}\},\\
V_2&=\{(z_1, z_2)\in S^3\mid |z_2|\geq \frac{\sqrt{2}}{2}\}.
\end{align*}
Let $D_{1,m}$ and $D_{2,n}$ be following meridian disks in $V_1$ and $V_2$ respectively:
\begin{align*}
D_{1,m}&=\{(\sqrt{1-r^2}e^{\frac{\pi}{4}i+\frac{m\pi}{2}i}, re^{\theta i})\in S^3\mid 0\leq r\leq\frac{\sqrt{2}}{2}, \theta\in \mathbb{R}\},\\
D_{2,n}&=\{(re^{\theta i}, \sqrt{1-r^2}e^{\frac{\pi}{2g+2}i+\frac{n\pi}{g+1}i})\in S^3\mid 0\leq r\leq\frac{\sqrt{2}}{2}, \theta\in \mathbb{R}\}.
\end{align*}
Here $m=0, 1, 2, 3$ and $n=0, 1, \cdots, 2g+1$. All the $D_{1,m}$ cut $V_1$ into $4$ cylinders. And each cylinder contains exactly one $a_m$ for some $m$. All the $D_{2,n}$ cut $V_2$ into $2g+2$ cylinders. And each cylinder contains exactly one $b_n$ for some $n$. Hence $T$ and all these disks cut $S^3$ into $2g+6$ cylinders. Let $V_g^c$ be the union of cylinders intersecting $\Gamma_g^c$, and ${V_g^c}'$ be the union of cylinders intersecting ${\Gamma_g^c}'$. Then $V_g^c$ and ${V_g^c}'$ are two handlebodies. $\partial V_g^c=\partial {V_g^c}'$ is a Heegaard surface of $S^3$, and $\partial V_g^c\cong\Sigma_g$.

Following three isometries on $S^3$ preserve the graph $\Gamma_g^c\cup{\Gamma_g^c}'$, the torus $T$, and the union of the disks. They also preserve $\partial V_g^c$.
\begin{align*}
\tau_g: (z_1, z_2)&\mapsto(iz_1, e^{\frac{\pi}{g+1}i}z_2)\\
\rho: (z_1, z_2)&\mapsto(-z_1, z_2)\\
\sigma: (z_1, z_2)&\mapsto(\bar{z}_1, z_2)
\end{align*}

For even $g>1$, $\tau_g$ has order $4g+4$ and $\tau_g^2\rho\sigma$ has order $2g+2$.

For each $g>1$, $\tau_g\sigma$ has order $2g+2$ and $\langle\tau_g\sigma, \rho\rangle$, $\langle\tau_g, \rho\rangle$,
$\langle\tau_g^2, \rho, \sigma\rangle$ are all abelian groups of
order $4g+4$.

Notice that $\tau_g$ and $\rho$ preserve the orientation of $S^3$, and $\sigma$ reverses the orientation of $S^3$. Only $\tau_g$ changes the two sides of $\partial V_g^c$. Then combining with Lemma \ref{Lem of change sides}, we have extendable cyclic (abelian) group actions on $\Sigma_g\cong \partial V_g^c$ as following:

(1). For even $g>1$, $\langle\tau_g\rangle$ gives a type $(-,+)$ extendable cyclic group action of order $4g+4$. This also realizes the maximum order of cyclic group actions on $\Sigma_g$ when $g$ is even.

(2). For even $g>1$, $\langle\tau_g^2\rho\sigma\rangle$ gives a type $(-,-)$ extendable cyclic group action of order $2g+2$.

(2$'$). For odd $g>1$, by adding shortest geodesics $a_0a_1$ and $a_0a_3$ to ${\Gamma_{g-1}^c}'$, we get a graph $\Gamma_g$. The group $\langle\tau_{g-1}^2\rho\sigma\rangle$ preserves $\Gamma_g$. Hence it also preserves some regular neighbourhood $N(\Gamma_g)$ of $\Gamma_g$. Then $\langle\tau_{g-1}^2\rho\sigma\rangle$ gives a type $(-,-)$ extendable cyclic group action of order $2g$ on $\Sigma_g\cong\partial N(\Gamma_g)$.

(3). For each $g>1$, $\langle\tau_g\sigma\rangle$ gives a type $(+,-)$ extendable cyclic group action of order $2g+2$.

(4). For each $g>1$, $\langle\tau_g\sigma, \rho\rangle$ gives a type $(+,-)$ extendable abelian group action of order $4g+4$. This also realizes the maximum order of orientation-preserving abelian group actions on $\Sigma_g$.

(5). For each $g>1$, $\langle\tau_g, \rho\rangle$ gives a type $(-,+)$ extendable abelian group action of order $4g+4$.

(6). For each $g>1$, $\langle\tau_g^2, \rho, \sigma\rangle$ gives a type $(-,-)$ extendable abelian group action of order $4g+4$.
\end{example}

\begin{example}[Wheel]\label{Ex of wheel}
Given odd $g>1$, we will construct a cyclic group
action on a Heegaard splitting of $S^3$.

Let $S^3$ and $T$ be in $\mathbb{C}^2$ as above. Let $L=L_1\cup L_2$ be the $(2, 4)$-torus link in $T$:
\begin{align*}
L_1&=\{(\frac{\sqrt{2}}{2}e^{2\theta i}, \frac{\sqrt{2}}{2}e^{\theta i})\in S^3\mid \theta\in\mathbb{R}\},\\
L_2&=\{(\frac{\sqrt{2}}{2}e^{2\theta i}, \frac{\sqrt{2}}{2}ie^{\theta i})\in S^3\mid \theta\in\mathbb{R}\}.
\end{align*}
Let
\begin{align*}
a_m&=(\frac{\sqrt{2}}{2}e^{\frac{2m\pi}{g-1}i}, \frac{\sqrt{2}}{2}e^{\frac{2m\pi}{2g-2}i}), m=0, 1, \cdots, 2g-3,\\
b_n&=(\frac{\sqrt{2}}{2}e^{\frac{(2n+1)\pi}{g-1}i}, \frac{\sqrt{2}}{2}ie^{\frac{(2n+1)\pi}{2g-2}i}), n=0, 1, \cdots, 2g-3.
\end{align*}
Then $a_0, a_1, \cdots, a_{2g-3}$ are $2g-2$ points in $L_1$, and $b_0, b_1, \cdots, b_{2g-3}$ are $2g-2$ points in $L_2$. Connect $a_i$ to $a_{i+g-1}$ with the shortest geodesic in $S^3$, and connect $b_i$ to $b_{i+g-1}$ with the shortest geodesic in $S^3$, here $i=0, 1, \cdots, g-2$. Then we get two graphs $\Gamma_g^w, {\Gamma_g^w}'\in S^3$. $\Gamma_g^w$ contains $L_1$ and ${\Gamma_g^w}'$ contains $L_2$. Each graph has $2g-2$ vertices and $3g-3$ edges. They are in the dual positions as in Figure \ref{fig of wheels} ($g=3$), here graphs have been projected into the three dimensional Euclidean space $E^3$ from $S^3-\{(-1, 0)\}$. The middle point of $a_1a_3$ is at the infinity.

\begin{figure}[h]
\centerline{\scalebox{0.67}{\includegraphics{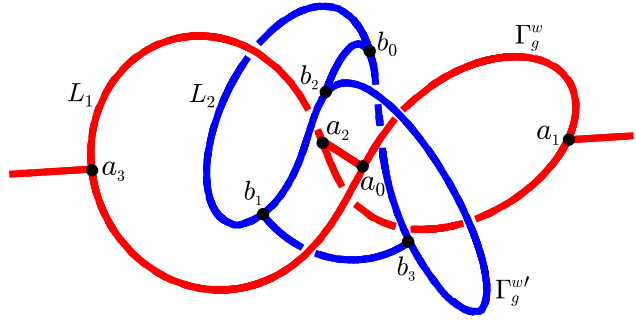}}}
\caption{$\Gamma_g^w$ and ${\Gamma_g^w}'$ in $S^3$}\label{fig of wheels}
\end{figure}

Let $T'$ be the following torus in $S^3$:
$$T'=\{(z_1, z_2)\in S^3\mid |z_1|=\frac{\sqrt{3}}{2}, |z_2|=\frac{1}{2}\}.$$
It splits $S^3$ into two solid tori ${V_1}'$ and ${V_2}'$:
\begin{align*}
{V_1}'&=\{(z_1, z_2)\in S^3\mid |z_1|\geq\frac{\sqrt{3}}{2}\},\\
{V_2}'&=\{(z_1, z_2)\in S^3\mid |z_2|\geq\frac{1}{2}\}.
\end{align*}
Let $D_k$ be the following meridian disk in ${V_1}'$:
$$D_k=\{(\sqrt{1-r^2}e^{\frac{\pi}{2g-2}i+\frac{k\pi}{g-1}i}, re^{\theta i})\in S^3\mid 0\leq r\leq \frac{1}{2}, \theta\in \mathbb{R}\}.$$
It lies between $(e^{\frac{k\pi}{g-1}i}, 0)$ and $(e^{\frac{(k+1)\pi}{g-1}i}, 0)$, here $k=0, 1, \ldots, 2g-3$.
Then these disks cut ${V_1}'$ into $2g-2$ cylinders. And each cylinder intersects exactly one of $\Gamma_g^w$ and ${\Gamma_g^w}'$. Let $A$ be the annulus in ${V_2}'$ separating $L_1$ and $L_2$:
$$A=\{(re^{2\theta i}, \sqrt{1-r^2}e^{\frac{\pi}{4}i+\theta i})\in S^3\mid -\frac{\sqrt{3}}{2}\leq r\leq\frac{\sqrt{3}}{2}, \theta\in\mathbb{R}\}.$$
Then $A$ cuts ${V_2}'$ into $2$ solid tori. Let $V_g^w$ be the
union of the solid torus and cylinders intersecting $\Gamma_g^w$, and ${V_g^w}'$ be the union of the solid torus and cylinders intersecting ${\Gamma_g^w}'$. Then $V_g^w$ and ${V_g^w}'$ are two handlebodies. $\partial V_g^w=\partial {V_g^w}'\cong\Sigma_g$, and $\partial V_g^w$ is a Heegaard surface of $S^3$.

Following isometry on $S^3$ preserves the graph $\Gamma_g^w\cup{\Gamma_g^w}'$, the torus $T'$, the annulus $A$, and the union of the disks. It also preserves $\partial V_g^w$.
$$\varphi_g: (z_1, z_2)\mapsto(e^{\frac{\pi}{g-1}i}z_1, ie^{\frac{\pi}{2g-2}i}z_2)$$

For odd $g>1$, $\varphi_g$ has order $4g-4$. Notice that $\varphi_g$ preserves the orientation of $S^3$ and changes the two sides of $\partial V_g^w$. Combining with Lemma \ref{Lem of change sides}, for odd $g>1$, $\langle\varphi_g\rangle$ gives a type $(-,+)$ extendable cyclic group action of order $4g-4$ on $\Sigma_g\cong \partial V_g^w$. This also realizes the maximum order of orientation-reversing periodic maps on $\Sigma_g$ when $g$ is odd.
\end{example}

\begin{example}[Fork]\label{Ex of fork}
Given even $g>1$, we will construct an abelian
group action on a Heegaard splitting of $S^3$.

Let $S^3$ be the unit sphere in $\mathbb{C}^2$:
$$S^3=\{(z_1, z_2)\in\mathbb{C}^2\mid |z_1|^2+|z_2|^2=1\}.$$
Let
\begin{align*}
a_m&=(e^{\frac{(2m+1)\pi}{2}i}, 0), m=0, 1,\\
b_n&=(0, e^{\frac{2n\pi}{g+2}i}), n=0, 1, \cdots, g+1.
\end{align*}
Connect $a_0$ to each $b_{2k}$ with the shortest geodesics in $S^3$, and connect $a_1$ to each $b_{2k+1}$ with the shortest geodesics in $S^3$, here $k=0, 1, \cdots, g/2$. Then we get two graphs $\Gamma^f_g, {\Gamma^f_g}'\in S^3$.  $\Gamma^f_g$ contains $a_0$, $b_{2k}$, and ${\Gamma^f_g}'$ contains $a_1$, $b_{2k+1}$. Let $S$ be the following two dimensional sphere in $S^3$:
$$S=\{(z_1, z_2)\in\mathbb{C}^2\mid z_1\in \mathbb{R}\}\cap S^3.$$
It cuts $S^3$ into $2$ three dimensional balls $B_0$ and $B_1$. $B_0$ contains $a_0$ and $B_1$ contains $a_1$. $\Gamma^f_g$, ${\Gamma^f_g}'$ and $S$ are shown in Figure \ref{fig of forks} ($g=4$), here they have been projected into the three dimensional Euclidean space $E^3$ from $S^3-\{(-1, 0)\}$.

\begin{figure}[h]
\centerline{\scalebox{0.385}{\includegraphics{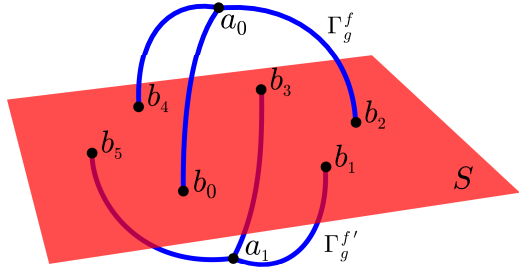}}}
\caption{$\Gamma^f_g$, ${\Gamma^f_g}'$ and $S$ in $S^3$}\label{fig of forks}
\end{figure}

Following two isometries on $S^3$ preserve the sphere $S$ and the graph $\Gamma^f_g\cup{\Gamma^f_g}'$.
\begin{align*}
\tau_{g+1}^2: (z_1, z_2)&\mapsto(-z_1, e^{\frac{2\pi}{g+2}i}z_2)\\
\rho\sigma: (z_1, z_2)&\mapsto(-\bar{z}_1, z_2)
\end{align*}
Hence they also preserve some closed regular neighbourhood $N(\Gamma^f_g)\cup N({\Gamma^f_g}')$ of $\Gamma^f_g\cup{\Gamma^f_g}'$, and preserve the common boundary of the two handlebodies $V^f_g$ and ${V^f_g}'$:
\begin{align*}
V^f_g&=N(\Gamma^f_g)\cup \overline{B_1-N({\Gamma^f_g}')},\\
{V^f_g}'&=N({\Gamma^f_g}')\cup \overline{B_0-N(\Gamma^f_g)}.
\end{align*}
The surface $\partial V^f_g=\partial {V^f_g}'$ is a Heegaard surface of $S^3$, and $\partial V^f_g\cong\Sigma_g$.

Since $\tau_{g+1}^2\rho\sigma=\rho\sigma\tau_{g+1}^2$, $\langle\tau_{g+1}^2, \rho\sigma\rangle$ is an abelian group of order $2g+4$. Since $\tau_{g+1}^2$ preserves the orientation of $S^3$, $\rho\sigma$ reverses the orientation of $S^3$, and only $\tau_{g+1}^2$ changes the two sides of $\partial V^f_g$, combining with Lemma \ref{Lem of change sides}, one can check that $\tau_{g+1}^2\rho\sigma$ and $\tau_{g+1}^2$ satisfy the condition of type $(Mix)$ extendable group action. Hence for even $g>1$, $\langle\tau_{g+1}^2, \rho\sigma\rangle$ gives a type $(Mix)$ extendable abelian group action of order $2g+4$ on $\Sigma_g\cong \partial V^f_g$.
\end{example}

\begin{example}[Square]\label{Ex of Square}
We will construct a $\oplus^5_{i=1}\mathbb{Z}_2$-action on $V_5$.

Let $\Gamma^s$ be the boundary of the square $[0,1]^2$ in $xy$-plane in the three dimensional Euclidean space $E^3$. Let $T_r$ and ${T_r}'\subset T_r$ be following translation groups:
\begin{align*}
T_r&=\{(a, b, c)\mid a, b, c\in \mathbb{Z}\},\\
{T_r}'&=\langle(2, 0, 0), (0, 2, 0), (0, 0, 1)\rangle.
\end{align*}
An element $t=(a, b, c)\in T_r$ acts on $E^3$ as following:
$$t: (x, y, z)\mapsto(x+a, y+b, z+c).$$
Following three isometries $r_x$, $r_y$ and $R_z$ on $E^3$ preserve the graph $\bigcup_{t\in T_r}t(\Gamma^s)$.
\begin{align*}
r_x: (x, y, z)&\mapsto(x, -y, -z)\\
r_y: (x, y, z)&\mapsto(-x, y, -z)\\
R_z: (x, y, z)&\mapsto(x, y, -z)
\end{align*}
In the three dimensional torus $E^3/{T_r}'\cong T^3$ we have a graph $\Gamma^s_5=(\bigcup_{t\in T_r}t(\Gamma^s))/{T_r}'$. It has $4$ vertices and $8$ edges. Since ${T_r}'$ is a normal subgroup of $\langle T_r, r_x, r_y, R_z\rangle$, the quotient group $\langle T_r, r_x, r_y, R_z\rangle/{T_r}'\cong\oplus^5_{i=1}\mathbb{Z}_2$ acts on $E^3/{T_r}'$, and it preserves $\Gamma^s_5$. Then it also preserves some closed regular neighbourhood $N(\Gamma^s_5)\cong V_5$ of $\Gamma^s_5$.

We see this $\oplus^5_{i=1}\mathbb{Z}_2$-action on $V_5$ is ``extendable over $T^3$''. But even on $\partial V_5\cong\Sigma_5$, it is not extendable over $S^3$, by Theorem \ref{Thm of maximum of CandA action}.
\end{example}

\section{Preliminaries for the proofs}\label{Sec of preliminaries}

In this section we will give some known results about actions of discrete groups on surfaces and three manifolds, which will be used later. Most results will be presented in the language of orbifold.

\begin{remark}
For orbifold theories, one can see \cite{BMP,MMZ,Th,Zi1}. In the following, contents in \S\ref{Subsec of OT} can be found in \cite{BMP}. Lemma \ref{Lem of 2Disorb} can be found in \cite{Ei}; Lemma \ref{Lem of 3Disorb} and \ref{Lem of 3SphCyc} rely on \cite{BMP,BP,Li,Sm,Th}: \cite{Li} for the orbifold having isolated singular points, and \cite{BMP,BP,Sm,Th} for the orbifold whose singular set has dimension at least $1$; Lemma \ref{Lem of Hanorb}, \ref{Lem of Euler char equ} and \ref{Lem of handcover} can be found in \cite{MMZ,Zi1}; Lemma \ref{Lem of extend to Hand} can be found in \cite{RZ2}.

\end{remark}


\subsection{Basic concepts in orbifold theory}\label{Subsec of OT}

Let $M$ be a n-dimensional manifold with a faithful smooth action of a discrete group $G$, then $M/G$ is an n-orbifold. For any $x\in M/G$, let $x'$ be one of its pre-images in $M$, and $St(x')$ be the stable subgroup of $x'$. Then the isomorphic type of $St(x')$ does not depend on the choice of $x'$. Denote it by $G_x$. It is the local group of $M/G$ at $x$. The order of $G_x$ is the index of $x$. If $|G_x|>1$, $x$ is a singular point of $M/G$. Otherwise $x$ is a regular point of $M/G$. If we forget all local groups and indices, then we get the underlying space $|M/G|$, which is a topological space with the quotient topology.

The orbifold $M/G$ is orientable if $M$ is orientable and $G$ preserves the orientation of $M$. $M/G$ is connected if $|M/G|$ is connected. And $M/G$ is compact if $|M/G|$ is compact. $M/G$ has non-empty boundary $\partial M/G$ if $M$ has non-empty boundary $\partial M$, then $\partial M/G$ is the orbifold $(\partial M)/G$. $M/G$ is closed if it is connected, compact and has empty boundary.

Two orbifolds $M_1/G_1$ and $M_2/G_2$ are homeomorphic via $p:M_1/G_1\rightarrow M_2/G_2$ if the induced map $|M_1/G_1|\rightarrow |M_2/G_2|$ is a homeomorphism, $p$ preserves local groups, and at any point $p$ can be locally lifted to an equivariant homeomorphism between open sets of $M_1$ and $M_2$.

If $M/G\cong M'/G'$ and $M'$ is simply connected, then the fundamental group $\pi_1(M/G)\cong G'$. Covering spaces of $M/G$ can be defined, and there is a one to one correspondence between covering spaces and conjugacy classes of subgroups of $\pi_1(M/G)$. Regular covering spaces will correspond to normal subgroups. A similar Van-Kampen theorem is also valid.

\begin{definition}\label{Def of isoP}
Let $P_T$ (respectively $P_O$, $P_I$) be a regular tetrahedron (respectively octahedron, icosahedron) centered at $(0,0,0)$ in $E^3$. For a polyhedron $P$ in $E^3$, define $I^+(P)$ to be the orientation-preserving isometric group of $P$.

For $n\in \mathbb{Z}_+$, define $r_n$ to be the following isometry on $E^3$:
$$r_n: (x,y,z) \mapsto (x\cos\frac{2\pi}{n}+y\sin\frac{2\pi}{n}, -x\sin\frac{2\pi}{n}+y\cos\frac{2\pi}{n}, z).$$


\end{definition}


\begin{lemma}\label{Lem of loc2orb}
Let $x$ be a point in an orientable 2-orbifold, then $x$ has a neighbourhood which is homeomorphic to some $E^2/\langle r_n\rangle$, and $G_x\cong \langle r_n\rangle$, $n\in \mathbb{Z}_+$. Here $E^2$ is the $xy$-plane in $E^3$.

\end{lemma}

\begin{lemma}\label{Lem of loc3orb}
Let $x$ be a point in an orientable 3-orbifold, then $x$ has a neighbourhood which is homeomorphic to some $E^3/H$, and $G_x\cong H$. $H$ is one of the groups $\langle r_n\rangle, \langle r_n, r_x\rangle, I^+(P_T), I^+(P_O), I^+(P_I), n\in \mathbb{Z}_+$. Here $r_x$ is defined in Example \ref{Ex of Square}


\begin{figure}[h]
\centerline{\scalebox{0.6}{\includegraphics{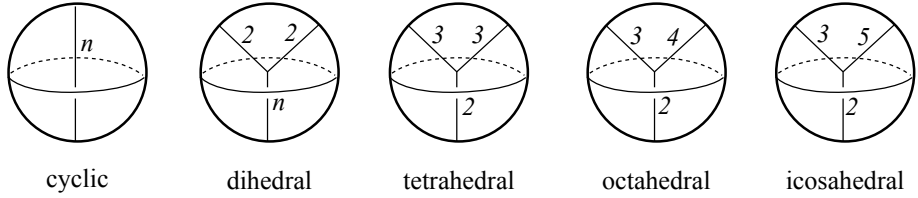}}}
\caption{Local models of 3-orbifolds}\label{fig of 3orbiLoc}
\end{figure}
\end{lemma}

\subsection{Results about discal, spherical and handlebody orbifolds}

\begin{definition}\label{Def of DisSphHan}
Let $B^n$ and $S^n$ be the n-ball and n-sphere respectively. An orbifold which is homeomorphic to some $B^n/G$ (respectively $S^n/G$, $V_g/G$) is called a discal n-orbifold (respectively spherical n-orbifold, handlebody orbifold).
\end{definition}

\begin{lemma}\label{Lem of 2Disorb}
An orientable discal 2-orbifold is homeomorphic to some $B^2_u/\langle r_n\rangle$. Here $B^2_u$ is the unit 2-ball centered at $(0,0,0)$ in the $xy$-plane.

\end{lemma}

\begin{lemma}\label{Lem of 3Disorb}
An orientable discal 3-orbifold is homeomorphic to some $B^3_u/H$, and an orientable spherical 2-orbifold is homeomorphic to some $S^2_u/H$. Here $B^3_u$ and $S^2_u$ are the unit 3-ball and 2-sphere centered at $(0,0,0)$ in $E^3$ respectively. $H$ is one of $\langle r_n\rangle$, $\langle r_n, r_x\rangle$, $I^+(P_T)$, $I^+(P_O)$, $I^+(P_I)$, $n\in \mathbb{Z}_+$.

\end{lemma}

\begin{lemma}\label{Lem of 3SphCyc}
Let $S^3/G$ be a spherical orbifold, here $G$ is a non-trivial cyclic group.

If $S^3/G$ is orientable, then the set of singular points of index $|G|$ is $\emptyset$ or a $S^1$.

If $S^3/G$ is not orientable, then the set of singular points of index $|G|$ is a $S^0$ or a $S^2$. Here $S^0\cong \{0,1\}$ with discrete topology.
\end{lemma}

\begin{lemma}\label{Lem of Hanorb}
An orientable handlebody orbifold is a union of finitely many orientable discal 3-orbifolds $\{B^3_i/G_i\}^n_{i=1}, n\in \mathbb{Z}_+$, such that:

(a). In each $\partial B^3_i/G_i$ finitely many disjoint orientable discal 2-orbifolds are given;

(b). For $i\neq j$, if $B^3_i/G_i\cap B^3_j/G_j\neq\emptyset$, then $B^3_i/G_i\cap B^3_j/G_j$ is a union of finitely many orientable discal 2-orbifolds in (a);

(c). Each discal 2-orbifold in (a) is in exactly two elements in $\{B^3_i/G_i\}^n_{i=1}$.

(d). There is a choice of orientations of $\{B^3_i/G_i\}^n_{i=1}$, such that on any common discal 2-orbifold, the induced orientations are opposite.

Conversely, if a connected orientable 3-orbifold is a union of finitely many orientable discal 3-orbifolds $\{B^3_i/G_i\}^n_{i=1}$ satisfying (a--d), then it is an orientable handlebody orbifold. $\bigcup^n_{i=1}B^3_i/G_i$ is called a discal orbifold decomposition of it.
\end{lemma}

\begin{definition}\label{Def of graph of groups}
A finite graph of finite groups $\mathcal{G}$ is a labelled finite graph. Each vertex and edge is labelled by a finite group. If $G_v$ is a vertex adjacent to an edge $G_e$, then there is an injective homomorphism $\phi_{ev}:G_e\rightarrow G_v$. If we forget all labels, then we get its underlying graph $|\mathcal{G}|$.
\end{definition}

\begin{definition}\label{Def of induced GofGs}
Let $\mathcal{H}$ be an orientable handlebody orbifold. $\mathcal{H}=\bigcup^n_{i=1}B^3_i/G_i$ is a discal orbifold decomposition as in Lemma \ref{Lem of Hanorb}. Then $\mathcal{H}$ has an induced finite graph of finite groups (induced by the discal orbifold decomposition) defined as following: it contains $\{G_i\}^n_{i=1}$ as vertices; there is an edge $G$ between $G_i$ and $G_j$ if $B^3_i/G_i$ and $B^3_j/G_j$ share a common discal 2-orbifold $B^2/G$.
\end{definition}

\begin{definition}\label{Def of Euler char}
Let $\mathcal{G}$ be a finite graph of finite groups. $V$ and $E$ are its vertex set and edge set respectively. Define its Euler characteristic:
$$\chi(\mathcal{G})=\sum_{G_v\in V}\frac{1}{|G_v|}-\sum_{G_e\in E}\frac{1}{|G_e|}.$$

Let $\mathcal{H}$ be an orientable handlebody orbifold, and $\mathcal{H}\cong V_g/G$. Define its Euler characteristic:
$$\chi(\mathcal{H})= \frac{1-g}{|G|}.$$
\end{definition}

\begin{definition}\label{Def of finite injsurj}
A homomorphism $\phi: H\rightarrow G$ between groups is a finitely injective surjection if it is surjective and on any finite subgroup of $H$ it is injective.
\end{definition}

\begin{lemma}\label{Lem of Euler char equ}
Let $\mathcal{H}$ be an orientable handlebody orbifold. Let $\mathcal{G}$ be an induced finite graph of finite groups of $\mathcal{H}$. Then $\chi(\mathcal{H})=\chi(\mathcal{G})$.
\end{lemma}


\begin{lemma}\label{Lem of handcover}
Let $\mathcal{H}$ be an orientable handlebody orbifold, and $G$ be a finite group. Then $G$ can orientation-preservingly act on $V_g$ such that $V_g/G\cong \mathcal{H}$ if and only if there is a finitely injective surjection $\phi:\pi_1(\mathcal{H})\rightarrow G$ and $1-g=\chi(\mathcal{H})|G|$.
\end{lemma}

\begin{lemma}\label{Lem of extend to Hand}
Let $G$ be an abelian group acting on $\Sigma_g$. Suppose $\Sigma_g/G$ is orientable, having singular points $p_1,\cdots, p_{\alpha}$ and $q_1, \cdots, q_{\beta}$. Index of $p_j (1\leq j\leq \alpha)$ is $n_j$ and $n_j>2$; index of $q_k (1\leq k\leq \beta)$ is $2$. Then for some $\gamma\in \mathbb{Z}_+\cup\{0\}$ we have
\begin{align*}
\pi_1(\Sigma/G)=&\langle a_1,b_1,\cdots,a_{\gamma},b_{\gamma}, x_1,\cdots,x_{\alpha},y_1,\cdots,y_{\beta}\mid \\ &\prod^{\gamma}_{i=1}[a_i,b_i] \prod^{\alpha}_{j=1}x_j\prod^{\beta}_{k=1}y_k=1, x_j^{n_j}=y_k^2=1,1\leq j\leq \alpha,1\leq k\leq \beta\rangle.
\end{align*}
The action corresponds to a finitely injective surjection $\phi:\pi_1(\Sigma_g)\rightarrow G$. Then the following (a), (b) and (c) are equivalent:

(a). The $G$-action on $\Sigma_g$ extends to a handlebody $V_g$, with $\partial V_g=\Sigma_g$.

(b). The $G$-action on $\Sigma_g$ extends to a compact 3-manifold $M$, with $\partial M=\Sigma_g$.

(c). The generators $x_1,\cdots,x_\alpha$ corresponding to $p_1,\cdots, p_{\alpha}$ can be partitioned into pairs $x_s$, $x_t$ such that $\phi(x_s)=\phi(x_t^{-1})$.
\end{lemma}

\section{Maximum orders of cyclic and abelian actions on handlebodies}\label{Sec of maxOrdHan}

In this section we will determine the maximum orders of cyclic and abelian group actions on orientable handlebodies. Firstly we need some lemmas.

\subsection{Some useful lemmas}

\begin{lemma}\label{Lem of hand classify}
Let $V_g/G$ be an orientable handlebody orbifold. Here $G$ is cyclic and $|G|>g-1$. Then $V_g/G$ has an induced finite graph of finite groups as one of the following, here $l,m,n>1$.

\xymatrix{(A).&\mathbb{Z}_l \ar@{-}[r] & \mathbb{Z}_m \ar@{-}[r] & \mathbb{Z}_n&(1/l+1/m+1/n>1)}
\xymatrix{(B).&\mathbb{Z}_m \ar@{-}[r] & \mathbb{Z}_n& (0<1/m+1/n<1)}
\xymatrix{(C).&\mathbb{Z}_n \ar@(ur,dr)@{-}[]}
\xymatrix{(D).&\mathbb{Z}_n \ar@{-}[r] & \mathbb{Z}_m \ar@(ur,dr)@{-}[]^{\mathbb{Z}_m}}
\end{lemma}

\begin{proof}
Let $V_g/G=\bigcup^n_{i=1}B^3_i/G_i$ as in Lemma \ref{Lem of Hanorb}, and $\mathcal{G}$ be the induced finite graph of finite groups. Since $|G|>g-1$, by Definition \ref{Def of Euler char} and Lemma \ref{Lem of Euler char equ}, we have: $$\chi(\mathcal{G})=\chi(V_g/G)=\frac{1-g}{|G|}\in (-1,0).$$

Since $G$ is cyclic, all the vertices and edges of $\mathcal{G}$ are cyclic groups. Hence if an edge $G_e$ is non-trivial, then any vertex of it must be isomorphic to it, and if a vertex $G_v$ is non-trivial, then there are at most two non-trivial edges adjacent to it. Let $V$ and $E$ be the vertex set and edge set of $\mathcal{G}$ respectively, and $V'$ and $E'$ be the non-trivial vertex set and edge set of $\mathcal{G}$ respectively. By Definition \ref{Def of Euler char}, we have:
\begin{align*}
\chi(\mathcal{G})&=\chi(|\mathcal{G}|)+\sum_{G_v\in V}(\frac{1}{|G_v|}-1)-\sum_{G_e\in E}(\frac{1}{|G_e|}-1)\\ &=\chi(|\mathcal{G}|)+\sum_{G_v\in V'}(\frac{1}{|G_v|}-1)-\sum_{G_e\in E'}(\frac{1}{|G_e|}-1)\\
&<\chi(|\mathcal{G}|).
\end{align*}

Hence $\chi(|\mathcal{G}|)\in \{0,1\}$.

If $\chi(|\mathcal{G}|)=1$, then $|\mathcal{G}|$ is a tree. If $\mathcal{G}$ contains a trivial vertex $G_i$, and $G_e$ is an edge adjacent to $G_v$, $G_j$ is the other vertex of $G_e$, then $G_e$ is trivial; if $\mathcal{G}$ contains a non-trivial edge $G_e$, and $G_i$, $G_j$ are its two vertices, then $G_e\cong G_i\cong G_j$. In each case $B^3_i/G_i\cup B^3_j/G_j\cong B^3_j/G_j$. Replacing $B^3_i/G_i$ and $B^3_j/G_j$ by their union, we get a new discal decomposition of $V_g/G$. The new induced finite graph of finite groups will have less trivial vertices or less non-trivial edges. Hence we can assume $V'=V$ and $E'=\emptyset$. Then
$$\chi(\mathcal{G})=1+\sum_{G_v\in V'}(\frac{1}{|G_v|}-1).$$
Since for $G_v\in V'$ we have $|G_v|>1$, $|V'|=|V|$ must be $2$ or $3$. Hence $\mathcal{G}$ will be some finite graph of finite groups in the class (A) or (B).

If $\chi(|\mathcal{G}|)=0$, then we can delete an edge $G_e$ of $\mathcal{G}$ to get a maximal tree $\mathcal{G}'$ of it. Similar to the above discussion, we can assume $V'=V$ and edges in $\mathcal{G}'$ are trivial. The edge $G_e$ may be trivial or non-trivial. Then
$$\chi(\mathcal{G})=\sum_{G_v\in V'}(\frac{1}{|G_v|}-1)-(\frac{1}{|G_e|}-1).$$
If $G_e$ is trivial, then $|V'|=|V|=1$. Hence $\mathcal{G}$ will be some finite graph of finite groups in the class (C). If $G_e$ is non-trivial, then $|V'|=|V|=2$. If $G_e$ is a loop, then $\mathcal{G}$ will be some finite graph of finite groups in the class (D). Otherwise, as above we can get a new $\mathcal{G}$ with less non-trivial edges. Then $\mathcal{G}$ will be some finite graph of finite groups in the class (C).
\end{proof}

Let $\mathcal{C}$ be one of the classes (A--D) in Lemma \ref{Lem of hand classify}. In the following, we say that an orientable handlebody orbifold $\mathcal{H}$ is in the class $\mathcal{C}$ if $\mathcal{H}$ has an induced finite graph of finite groups belonging to $\mathcal{C}$. Following Lemma \ref{Lem of FundG} and \ref{Lem of EulerNo} can be derived easily from the Van-Kampen theorem and Definition \ref{Def of Euler char}.

\begin{lemma}\label{Lem of FundG}
The fundamental groups of the handlebody orbifolds in classes (A--D) are given in Table \ref{tab of FundG}.
\begin{table}[h]
\caption{Fundamental groups of classes (A--D)}\label{tab of FundG}
  \begin{tabular}{|l|l|}
  \hline (A) & $\mathbb{Z}_l\ast\mathbb{Z}_m\ast\mathbb{Z}_n$\\
  \hline (B) & $\mathbb{Z}_m\ast\mathbb{Z}_n$\\
  \hline (C) & $\mathbb{Z}_n\ast\mathbb{Z}$\\
  \hline (D) & $\mathbb{Z}_n\ast (\mathbb{Z}_m\oplus\mathbb{Z})$ \\\hline
  \end{tabular}
\end{table}
\end{lemma}

\begin{lemma}\label{Lem of EulerNo}
The Euler characteristic of the finite graphs of finite groups in classes (A--D) are given in Table \ref{tab of EulerNo}.
\begin{table}[h]
\caption{Euler characteristic of classes (A--D)}\label{tab of EulerNo}
  \begin{tabular}{|l|l|}
  \hline (A) & $1/l+1/m+1/n-2$\\
  \hline (B) & $1/m+1/n-1$\\
  \hline (C) & $1/n-1$\\
  \hline (D) & $1/n-1$\\\hline
  \end{tabular}
\end{table}
\end{lemma}

\begin{lemma}\label{Lem of odd order}
Let $\langle h\rangle$ be a finite cyclic group acting faithfully on a manifold $M$. $\overline{h}$ is the image of $h$ in $\langle h\rangle/\langle h^2\rangle$. Then $\overline{h}$ acts on $M/\langle h^2\rangle$. Suppose $x$ is a fixed point of $\overline{h}$ and $x$ has index $|G_x|$ in $M/\langle h^2\rangle$, then $|\langle h^2\rangle|/|G_x|$ is odd.
\end{lemma}

\begin{proof}
Let $x'$ be a pre-image of $x$ in $M$. Since $x$ is a fixed point of $\overline{h}$, the orbit $\{h^i(x')\mid i\in\mathbb{Z}\}$ and $\{h^{2i}(x')\mid i\in\mathbb{Z}\}$ are the same. Hence the pre-image of $x$ in $M$ contains odd points. Namely $|\langle h^2\rangle|/|G_x|$ is odd.
\end{proof}

\begin{lemma}\label{Lem of fixed point}
Let $h$ be a periodic map of order $2$ on a compact manifold $M$. If the Euler characteristic of $M$ is odd, then $h$ has a fixed point.
\end{lemma}

\begin{proof}
Since $h$ has order $2$, the induced linear map $h_{*i}:H_i(M,\mathbb{R})\rightarrow H_i(M,\mathbb{R})$ can only have eigenvalues $\pm 1$. Let $L(h)$ be the Lefschetz number of $h$, and $\chi(M)$ be the Euler characteristic of $M$. Then $L(h)\equiv\chi(M) (mod\, 2)$. Since $\chi(M)$ is odd, $L(h)\neq 0$. Then by the Lefschetz fixed point theorem, $h$ has a fixed point.
\end{proof}

\begin{lemma}\label{Lem of Tocyclic}
Let $l,m,n\in\mathbb{Z}_+$, and for $1\leq i\leq l$, $m_i \in\mathbb{Z}_+$. Let $[m_1,m_2,\cdots ,m_l]$ be the lowest common multiple of $m_i$, $1\leq i\leq l$.

(1). If $\phi:\mathbb{Z}_{m_1}\ast\mathbb{Z}_{m_2}\ast\cdots \ast\mathbb{Z}_{m_l}\rightarrow \mathbb{Z}_n$ is a finitely injective surjection, then $n=[m_1,m_2,\cdots ,m_l]$.

(2). If $\phi:\mathbb{Z}_{m_1}\ast\mathbb{Z}_{m_2}\ast\cdots \ast\mathbb{Z}_{m_l}\ast(\mathbb{Z}_{m}\oplus\mathbb{Z})\rightarrow \mathbb{Z}_n$ is a finitely injective surjection, then $n=t[m_1,m_2,\cdots, m_l,m]$ for some $t\in \mathbb{Z}_+$.
\end{lemma}

\begin{proof}
(1). The homomorphism $\phi$ induces a surjective homomorphism
$$\phi':\mathbb{Z}_{m_1}\oplus\mathbb{Z}_{m_2}\oplus\ldots \oplus\mathbb{Z}_{m_l}\rightarrow \mathbb{Z}_n,$$
which is injective on $\mathbb{Z}_{m_i}$, $1\leq i\leq l$. For any prime number $p$, let $\mathbb{Z}_{p^{\alpha_i}}$ be the $p$-component of $\mathbb{Z}_{m_i}$, and $\mathbb{Z}_{p^{\beta}}$ be the $p$-component of $\mathbb{Z}_n$. Since $\phi'$ is surjective, we have $max\{\alpha_1,\alpha_2,\cdots,\alpha_l\}\geq\beta$. Otherwise generators of $\mathbb{Z}_{p^{\beta}}$ will not lie in the image of $\phi'$. Since $\phi'$ is injective on $\mathbb{Z}_{m_i}$, we have $\beta\geq max\{\alpha_1,\alpha_2,\cdots,\alpha_l\}$. Hence $\beta=max\{\alpha_1,\alpha_2,\cdots,\alpha_l\}$, namely $n=[m_1,m_2,\cdots,m_l]$.

(2). The homomorphism $\phi$ induces a surjective homomorphism
$$\phi':\mathbb{Z}_{m_1}\oplus\mathbb{Z}_{m_2}\oplus\ldots \oplus\mathbb{Z}_{m_l}\oplus (\mathbb{Z}_{m}\oplus\mathbb{Z})\rightarrow \mathbb{Z}_n,$$
which is injective on $\mathbb{Z}_{m}$ and $\mathbb{Z}_{m_i}$, $1\leq i\leq l$. Let $\mathbb{Z}_{p^{\alpha}}$ be the $p$-component of $\mathbb{Z}_{m}$, then $\beta\geq max\{\alpha_1,\cdots,\alpha_l,\alpha\}$. Hence $n=t[m_1,\cdots,m_l,m]$ for some $t\in \mathbb{Z}_+$.
\end{proof}

\subsection{Maximum orders of cyclic and abelian actions on handlebodies}

\begin{proposition}\label{Pro of chg}
Table \ref{tab of maxorder of ORhandlebody} gives an upper bound of $CH_g^-$.
\end{proposition}

\begin{proof}
Let $\langle h\rangle$ be a cyclic group acting on $V_g$. $h$ reverses the orientation of $V_g$. Let $\overline{h}$ be the image of $h$ in $\langle h\rangle/\langle h^2\rangle$. Then $\overline{h}$ acts on $V_g/\langle h^2\rangle$ and has order $2$.


By (2) and (2$'$) in Example \ref{Ex of cage}, we have $CH_g^->2g-2$. If $|\langle h\rangle|=CH_g^-$, then $|\langle h^2\rangle|>g-1$. By Lemma \ref{Lem of hand classify}, $V_g/\langle h^2\rangle$ is in classes (A--D). Let $\Theta$ be the set of singular points in $V_g/\langle h^2\rangle$, and $N(\Theta)$ be a $\overline{h}$-invariant regular neighborhood of $\Theta$. Then $\overline{h}$ acts on $\Theta$ and $\overline{(V_g/\langle h^2\rangle)\backslash N(\Theta)}$ respectively.

{\bf Case 1}: $V_g/\langle h^2\rangle$ is in the class (A).

Then $|\langle h^2\rangle|$ is even. By Lemma \ref{Lem of fixed point}, $\overline{h}$ will have a fixed point in $\Theta$. Then by Lemma \ref{Lem of odd order}, for a fixed point $x$ we have that $|\langle h^2\rangle|/|G_x|$ is odd. Hence the index of $x$ must be even. Then we have $\{l,m,n\}=\{2,3,3\}$ or $\{2,2,2k\}$, $k\in \mathbb{Z}_+$.

If $\{l,m,n\}=\{2,3,3\}$, then by Lemma \ref{Lem of handcover}, there is a finitely injective surjection $\phi:\pi_1(V_g/\langle h^2\rangle)\rightarrow \langle h^2\rangle$. By Lemma \ref{Lem of FundG} and \ref{Lem of Tocyclic}, $\pi_1(V_g/\langle h^2\rangle)\cong \mathbb{Z}_2\ast\mathbb{Z}_3\ast\mathbb{Z}_3$ and $\langle h^2\rangle\cong \mathbb{Z}_6$. Hence $|\langle h\rangle|=12$. Then by Lemma \ref{Lem of EulerNo}, we have $g=6$.

Similarly, if $\{l,m,n\}=\{2,2,2k\}$, then $\langle h^2\rangle\cong \mathbb{Z}_{2k}$, $|\langle h\rangle|=4k$ and $g=2k$.

{\bf Case 2}: $V_g/\langle h^2\rangle$ is in the classes (B--D).

By Lemma \ref{Lem of fixed point}, $\overline{h}$ will have a fixed point $x$ in $\overline{(V_g/\langle h^2\rangle)\backslash N(\Theta)}$. By Lemma \ref{Lem of odd order}, $|\langle h^2\rangle|/|G_x|=|\langle h^2\rangle|$ is odd. Then by Lemma \ref{Lem of handcover}, \ref{Lem of FundG}, \ref{Lem of Tocyclic} and \ref{Lem of EulerNo}, we have:

If $V_g/\langle h^2\rangle$ is in the class (B), then $\pi_1(V_g/\langle h^2\rangle)\cong \mathbb{Z}_m\ast\mathbb{Z}_n$, and $\langle h^2\rangle\cong \mathbb{Z}_{[m,n]}$. Then $|\langle h\rangle|=2[m,n]$, and $g=[m,n]-(m+n)/(m,n)+1$. Since $n$ and $m$ are odd, $g$ is even.

If $V_g/\langle h^2\rangle$ is in the class (C), then $\pi_1(V_g/\langle h^2\rangle)\cong \mathbb{Z}_n\ast\mathbb{Z}$, and $\langle h^2\rangle\cong \mathbb{Z}_{kn}$. Then $|\langle h\rangle|=2kn$, and $g=kn-k+1$. Since $k$ and $n$ are odd, $g$ is odd.

If $V_g/\langle h^2\rangle$ is in the class (D), then $\pi_1(V_g/\langle h^2\rangle)\cong \mathbb{Z}_n\ast(\mathbb{Z}_m\oplus\mathbb{Z}$), and we have $\langle h^2\rangle\cong \mathbb{Z}_{k[m,n]}$. Then $|\langle h\rangle|=2k[m,n]$, and $g=k[m,n]-k[m,n]/n+1$. Since $k$, $m$ and $n$ are odd, $g$ is odd. Let $m=m'(m,n)$, then $|\langle h\rangle|=2(km')n$, and we have $g=(km')n-km'+1$.
\end{proof}

\begin{proposition}\label{Pro of chgcons}
Suppose $k,m,n\in \mathbb{Z}_+$. If $g=[m,n]-(m+n)/(m,n)+1$, then there is a cyclic group $\langle h\rangle$ of order $2[m,n]$ acting on $V_g$. And $h$ reverses the orientation of $V_g$.

If $g=kn-k+1$, then there is a cyclic group $\langle h\rangle$ of order $2kn$ acting on $V_g$. And $h$ reverses the orientation of $V_g$.
\end{proposition}

\begin{proof}
We can assume $m,n>1$. For $g=[m,n]-(m+n)/(m,n)+1$, let $\mathcal{F}$ be a 2-orbifold. It has underlying space $B^2$ and contains two singular points of indices $m$ and $n$ in the interior of $B^2$. Let $\mathcal{H}=\mathcal{F}\times [-1,1]$. It is in the class (B). $\pi_1(\mathcal{H})=\mathbb{Z}_m\ast\mathbb{Z}_n$, and there is a finitely injective surjection $\phi:\pi_1(\mathcal{H})\rightarrow \mathbb{Z}_{[m,n]}$. Hence by Lemma \ref{Lem of handcover}, there is a $\mathbb{Z}_{[m,n]}$-action on $V_g$ such that $V_g/\mathbb{Z}_{[m,n]}\cong \mathcal{H}$.

The pre-image of $\mathcal{F}\times \{0\}$ in $V_g$ is a surface $F$ and $V_g=F\times [-1,1]$. Then the product $Id_F\times (-Id_{[-1,1]})$ acts of $V_g$ and has order $2$. It is commutative with the $\mathbb{Z}_{[m,n]}$-action. Hence we have a cyclic action of order $2[m,n]$ on $V_g$. Its generator reverses the orientation of $V_g$.

For $g=kn-k+1$, the construction is similar. In this case $\mathcal{F}$ has underlying space an annulus and contains one singular point of index $n$ in its interior. $\mathcal{H}=\mathcal{F}\times [-1,1]$ is in the class (C).
\end{proof}

\begin{proposition}\label{Pro of ahg}
$AH_g^-=2AH_g$.
\end{proposition}

\begin{proof}
Clearly $AH_g^-\leq 2AH_g$. By (6) in Example \ref{Ex of cage} and Example \ref{Ex of Square}, we have $AH_g^-\geq 2AH_g$.
\end{proof}

\begin{proposition}\label{Pro of ag}
$A_g^-=AH_g^-$.
\end{proposition}

\begin{proof}
Let $G$ be an abelian group acting on $\Sigma_g$. $\Sigma_g/G$ is not orientable. Let $G_o$ be the subgroup of $G$ containing all elements which preserve the orientation of $\Sigma_g$. Then we have a 2-sheet cover $\Sigma_g/G_o\rightarrow\Sigma_g/G$, and the covering transformation $\tau$ reverses the orientation of $\Sigma_g/G_o$.

Let $\{p_1,\cdots, p_{\alpha}, q_1, \cdots, q_{\beta}\}$ be the set of singular points of $\Sigma_g/G_o$. $p_j (1\leq j\leq \alpha)$ has index bigger than $2$, and $q_k (1\leq k\leq \beta)$ has index $2$. Then $p_1, \cdots, p_{\alpha}$ can be partitioned into pairs by $\tau$. Since $G$ is abelian and $\tau$ reverses the orientation of $\Sigma_g/G_o$, their corresponding generators $x_1, \cdots, x_{\alpha}$ in $\Sigma_g/G_o$ can be partitioned into pairs satisfying the condition of Lemma \ref{Lem of extend to Hand}. Then $G_o$ can extend to a handlebody. Hence $|G_o|\leq AH_g$, and $A_g^-\leq AH_g^-$. Then $A_g^-=AH_g^-$.
\end{proof}

By Proposition \ref{Pro of chg}--\ref{Pro of ag} and Remark \ref{Rem of classical result}, we finished the proof of Proposition \ref{Pro of ORhandlebody}, hence finished the proof of Theorem \ref{Thm of CandA on handlebody}.

\section{Maximum orders of extendable actions on surfaces}\label{Sec of maxOrdSur}

In this section we will determine the maximum orders of extendable cyclic and abelian group actions on orientable closed surfaces. For an extendable action on $\Sigma_g$, we will identify $\Sigma_g$ and its image $e(\Sigma_g)$ in $S^3$. Firstly we will prove two lemmas.

\begin{lemma}\label{Lem of Fix circle}
Let $\langle h\rangle$ be a non-trivial cyclic group acting faithfully on $S^3$. $h$ preserves the orientation of $S^3$ and has a fixed point. Then the set of singular points in $S^3/\langle h\rangle$ is a $S^1$, with index $|\langle h\rangle|$, and $\pi_1(|S^3/\langle h\rangle|)$ is trivial.
\end{lemma} 

\begin{proof}
Let $k$ be a factor of $|\langle h\rangle|$ and $k\neq |\langle h\rangle|$. By Lemma \ref{Lem of 3SphCyc}, the set of fixed points of $\langle h^k\rangle$ is a $S^1$. Hence it belongs to the set of fixed points of $\langle h\rangle$. Namely the set of singular points in $S^3/\langle h\rangle$ is a $S^1$, with index $|\langle h\rangle|$. Since $\pi_1(S^3/\langle h\rangle)\cong\langle h\rangle$, $\pi_1(|S^3/\langle h\rangle|)$ must be trivial.
\end{proof}

\begin{lemma}\label{Lem of homology}
Let $S_H^3$ be an integer homology 3-sphere, and $S_g$ be a genus $g\geq 0$ orientable closed subsurface in $S_H^3$. Then $S_g$ splits $S_H^3$ into two compact manifolds $M$ and $N$, each has the same integer homology groups as $V_g$.
\end{lemma}

\begin{proof}
Since $S_H^3$ is orientable, $S_g$ is two sided in $S_H^3$. If it does not split $S_H^3$, then there is a surjective homomorphism from $H_1(S_H^3)$ to $\mathbb{Z}$, which is a contradiction. Then we have the following MV-sequence:

\centerline{\xymatrix{H_3(S_H^3) \ar[r] & H_2(S_g) \ar[r] & H_2(M)\oplus H_2(N) \ar[r] & H_2(S_H^3)}}

\centerline{\xymatrix{H_2(S_H^3) \ar[r] & H_1(S_g) \ar[r] & H_1(M)\oplus H_1(N) \ar[r] & H_1(S_H^3)}}

Clearly we have $H_i(M)=H_i(N)=0$ for $i>2$. By the first exact sequence, $H_2(M)=H_2(N)=0$. By the second exact sequence, $H_1(M)\oplus H_1(N)\cong \oplus_{i=1}^{2g}\mathbb{Z}$.
Then by the ``half live half die'' theorem, $H_1(M)\cong H_1(N)\cong \oplus_{i=1}^{g}\mathbb{Z}$.
\end{proof}

\subsection{Maximum orders of cyclic extendable actions on surfaces}

\begin{proposition}\label{Pro of cemp}
$CE_g(-,+)=C_g^-$.
\end{proposition}

\begin{proof}
Clearly $CE_g(-,+)\leq C_g^-$. By (1) in Example \ref{Ex of cage} and Example \ref{Ex of wheel}, we have $CE_g(-,+)=C_g^-$.
\end{proof}

\begin{proposition}\label{Pro of cepm}
$CE_g(+,-)=2g+2$.
\end{proposition}

\begin{proof}
Let $\langle h\rangle$ be a cyclic group acting on $\Sigma_g$. And the action is a type $(+,-)$ extendable action realizing the maximum order $CE_g(+,-)$. Then the action of $\langle h^2\rangle$ on $\Sigma_g$ can extend to some 3-manifold. Hence by Lemma \ref{Lem of extend to Hand}, it can extend to a handlebody $V_g$. By (3) in Example \ref{Ex of cage}, we have $|\langle h^2\rangle|>g-1$. Hence $V_g/\langle h^2\rangle$ is in the classes (A--D) in Lemma \ref{Lem of hand classify}.


Since $h$ reverses the orientation of $S^3$, it has a fixed point. Then $h^2$ has a fixed point. By Lemma \ref{Lem of Fix circle}, the set of singular points of $S^3/\langle h^2\rangle$ is a $S^1$, with index $|\langle h^2\rangle|$. Then $V_g/\langle h^2\rangle$ can only belong to classes (A--C). By Lemma \ref{Lem of EulerNo}, we have:

If $V_g/\langle h^2\rangle$ is in the class (A), then $|\langle h^2\rangle|=l=m=n=2$. Hence $|\langle h\rangle|=4$, and $g=2$.

If $V_g/\langle h^2\rangle$ is in the class (B), then $|\langle h^2\rangle|=m=n$. Hence $|\langle h\rangle|=2n$, $g=n-1$.

If $V_g/\langle h^2\rangle$ is in the class (C), then $|\langle h^2\rangle|=n$. Hence $|\langle h\rangle|=2n$, $g=n$.

Hence $|\langle h\rangle|\leq 2g+2$. Then by (3) in Example \ref{Ex of cage}, $CE_g(+,-)=2g+2$.
\end{proof}

\begin{proposition}\label{Pro of cemm}
$CE_g(-,-)=2g+1+(-1)^g$.
\end{proposition}

\begin{proof}
The proof is the same as above, except that if $V_g/\langle h^2\rangle$ is in the class (B).

In this case, $\Sigma_g/\langle h^2\rangle$ is a $S^2$ with four singular points. Let $\Theta$ be the set of singular points in $S^3/\langle h^2\rangle$. By Lemma \ref{Lem of Fix circle}, $\Theta$ is a $S^1$, and $S^3/\langle h^2\rangle$ is an integer homology 3-sphere. By Lemma \ref{Lem of homology}, $|\Sigma_g/\langle h^2\rangle|$ splits $|S^3/\langle h^2\rangle|$ into two compact manifolds $M$ and $N$. Let $\overline{h}$ be the image of $h$ in $\langle h\rangle/\langle h^2\rangle$, and $N(\Theta)$ be a $\overline{h}$-invariant regular neighbourhood of $\Theta$. Then $\overline{h}$ acts on $\overline{M-N(\Theta)}$. By Lemma \ref{Lem of homology}, $\overline{M-N(\Theta)}$ will have the same homology groups as $V_2$. Then by Lemma \ref{Lem of fixed point}, $\overline{h}$ has a fixed point in $\overline{M-N(\Theta)}$. By Lemma \ref{Lem of odd order}, we have that $|\langle h^2\rangle|$ is odd. Hence $g$ is even. Namely we have $CE_g(-,-)\leq 2g+1+(-1)^g$.

Then by (2) and (2$'$) in Example \ref{Ex of cage}, $CE_g(-,-)=2g+1+(-1)^g$.
\end{proof}

\begin{proposition}\label{Pro of cemix}
The type $(Mix)$ extendable cyclic group action does not exist.
\end{proposition}

\begin{proof}
For a type $(Mix)$ extendable $G$-action, there is a surjective homeomorphism $G\rightarrow\mathbb{Z}_2\oplus \mathbb{Z}_2$. Hence $G$ can not be cyclic.
\end{proof}

\subsection{Maximum orders of abelian extendable actions on surfaces}

\begin{proposition}\label{Pro of aealmost}
$AE_g(+,-)=AE_g(-,+)=AE_g(-,-)=2AE_g$.
\end{proposition}

\begin{proof}
Let $\mathcal{T}$ be one of $(+,-)$, $(-,+)$ and $(-,-)$. Clearly $AE_g\mathcal{T}\leq 2AE_g$. Then by (4), (5) and (6) in Example \ref{Ex of cage}, we have $AE_g\mathcal{T}=2AE_g$.
\end{proof}

\begin{proposition}\label{Pro of abelnot}
For odd $g>1$, the type $(Mix)$ extendable abelian group action does not exist.
\end{proposition}

\begin{proof}
Suppose there is a type $(Mix)$ extendable abelian $G$-action on $\Sigma_g$. Then there exist $h_1, h_2\in G$ such that $h_1$ preserves the orientation of $\Sigma_g$ and reverses the orientation of $S^3$, and $h_2\in G$ reverses the orientation of $\Sigma_g$ and preserves the orientation of $S^3$. Then both $h_1$ and $h_2$ change the two sides of $\Sigma_g$.

By Lemma \ref{Lem of homology}, $\Sigma_g$ splits $S^3$ into $M$ and $N$, and $H_1(M)\simeq H_1(N)\simeq \oplus_{i=1}^{g}\mathbb{Z}$. Let $i_{M}: \Sigma_g\rightarrow M$, $i_{N}: \Sigma_g\rightarrow N$ be the inclusions. We can choose a basis of $H_1(\Sigma_g)$,
denoted by $\{\alpha_1, \alpha_2, \cdots, \alpha_g, \beta_1, \beta_2, \cdots, \beta_g\}$, satisfying the following conditions:

(1) $\{{i_M}_*(\alpha_k)\}$ is a basis of $H_1(M)$, $k=1,2,\cdots,g$.

(2) $\{{i_N}_*(\beta_k)\}$ is a basis of $H_1(N)$, $k=1,2,\cdots,g$.

(3) ${i_N}_*(\alpha_k)=0$ and ${i_M}_*(\beta_k)=0$, $k=1, 2, \cdots,g$.

Then consider the intersection product of $H_1(\Sigma_g)$. By the
condition (3) we have $\alpha_i\bullet\alpha_j=0$ and
$\beta_i\bullet\beta_j=0$, $1\leq i, j\leq g$, hence we have:
\[ \begin{pmatrix}
\alpha_1\\ \vdots\\ \alpha_g\\ \beta_1\\ \vdots\\
\beta_g
\end{pmatrix} \bullet
  \begin{pmatrix}
\alpha_1, \cdots, \alpha_g, \beta_1, \cdots, \beta_g
\end{pmatrix} =
  \begin{pmatrix}
0 & X \\ -X^t & 0
\end{pmatrix} \]
Here $X=(x_{i,j})$ is a $g\times g$ matrix, and $x_{i,j}=\alpha_i\bullet\beta_j$. $X^t$ is the transposition of $X$. There are $g\times g$ matrices $A,B,C,D$, such that:
\[ {h_1}_*\begin{pmatrix}
\alpha_1, \cdots, \alpha_g, \beta_1, \cdots, \beta_g
\end{pmatrix} =
  \begin{pmatrix}
\alpha_1, \cdots, \alpha_g, \beta_1, \cdots, \beta_g
\end{pmatrix}
  \begin{pmatrix}
0 & A \\ B & 0
\end{pmatrix} \]
\[ {h_2}_*\begin{pmatrix}
\alpha_1, \cdots, \alpha_g, \beta_1, \cdots, \beta_g
\end{pmatrix} =
  \begin{pmatrix}
\alpha_1, \cdots, \alpha_g, \beta_1, \cdots, \beta_g
\end{pmatrix}
  \begin{pmatrix}
0 & C \\ D & 0
\end{pmatrix} \]
By computing the intersection product, we have the following:
\[ \begin{pmatrix}
0 & B^t \\ A^t & 0
\end{pmatrix}
  \begin{pmatrix}
0 & X \\ -X^t & 0
\end{pmatrix}
  \begin{pmatrix}
0 & A \\ B & 0
\end{pmatrix} =
  \begin{pmatrix}
0 & X \\ -X^t & 0
\end{pmatrix} \]
\[ \begin{pmatrix}
0 & D^t \\ C^t & 0
\end{pmatrix}
  \begin{pmatrix}
0 & X \\ -X^t & 0
\end{pmatrix}
  \begin{pmatrix}
0 & C \\ D & 0
\end{pmatrix} =
  \begin{pmatrix}
0 & -X \\ X^t & 0
\end{pmatrix} \]
Hence we have $A^tXB=-X^t$ and $C^tXD=X^t$. Since the group $G$ is abelian, $h_1h_2=h_2h_1$. Then we have:
\[ \begin{pmatrix}
0 & A \\ B & 0
\end{pmatrix}
  \begin{pmatrix}
0 & C \\ D & 0
\end{pmatrix} =
  \begin{pmatrix}
0 & C \\ D & 0
\end{pmatrix}
  \begin{pmatrix}
0 & A \\ B & 0
\end{pmatrix} \]
This is equivalent to $AD=CB$ and $BC=DA$.

The $g\times g$ matrices $A$, $B$, $C$, $D$ and $X$ are all non-degenerate. By computing the determinants we have
$|A||X||B|=(-1)^g|X|$, $|C||X||D|=|X|$ and $|A||D|=|C||B|$. Hence
$(-1)^g=|ABCD|=|AD|^2$, and $g$ is even.
\end{proof}

\begin{proposition}\label{Pro of aeeven}
For even $g>1$, $AE_g(Mix)=2g+4$.
\end{proposition}

\begin{proof}
Let $G$ be an abelian group acting on $\Sigma_g$. The action is a type $(Mix)$ extendable action, and $|G|=AE_g(Mix)$. Let $G_o$ be the subgroup of $G$ containing all elements preserving both the orientations of $\Sigma_g$ and $S^3$, then $|G|=4|G_o|$.

Choose $h\in G$ such that it reverses the orientation of $S^3$. Let $Fix(h)$ be the set of fixed points of $h$. By Lemma \ref{Lem of 3SphCyc}, $Fix(h)$ is a $S^0$ or $S^2$. Since $G$ is abelian, every element of $G$ keeps $Fix(h)$ invariant. Let $G_o'$ be the subgroup of $G$ containing all elements preserving both the orientations of $Fix(h)$ and $S^3$. If $Fix(h)\cong\{0, 1\}$, then preserving the orientation of $Fix(h)$ means fixing $0$ and $1$.

{\bf Case 1}: $G_o'\ncong\mathbb{Z}_2\oplus \mathbb{Z}_2$ and $|G_o'|>2$.

Then $G_o'$ is a cyclic group of order $|G_o'|\geq 3$, and the $G_o'$-action on $S^3$ has a fixed point. By Lemma \ref{Lem of Fix circle}, the set of singular points of $S^3/G_o'$ is a circle $\Theta$, with index $|G_o'|$, and $\pi_1(|S^3/G_o'|)$ is trivial. $|S^3/G_o'|$ is an integer homology 3-sphere.

Let $\Theta'$ be the pre-image of $\Theta$ in $S^3$, then $\Theta'$ is also a circle. Since $|G_o'|\geq 3$, $\Sigma_g$ intersects $\Theta'$ transversely. Hence $|\Sigma_g/G_o'|$ is an orientable closed surface in $|S^3/G_o'|$. Then $\Sigma_g/G_o'$ is an orientable 2-suborbifold in $S^3/G_o'$, namely $G_o'$ preserves both the orientations of $\Sigma_g$ and $S^3$. Since $|G/G_o'|\leq 4$, we have $G_o'=G_o$. By Lemma \ref{Lem of homology}, $|\Sigma_g/G_o|$ splits $|S^3/G_o|$ into two compact manifolds $M$ and $N$.

Let $g'$ be the genus of $|\Sigma_g/G_o|$, and $k$ be the number of singular points in $\Sigma_g/G_o$. Then $k$ must be even. By the Riemann-Hurwitz formula, we have:
$$2-2g=|G_o|(2-2g'-k(1-\frac{1}{|G_o|})).$$
By Example \ref{Ex of fork}, $|G_o|\geq (2g+4)/4=g/2+1$. Then we have:
$$(g',k,|G_o|)=(2,0,g-1),(1,2,g),(0,4,g+1),(0,6,g/2+1).$$
Notice that there is a $G/G_o\cong\mathbb{Z}_2\oplus \mathbb{Z}_2$-action on $(\Sigma_g/G_o,S^3/G_o)$. It satisfies a similar condition of the type $(Mix)$ extendable action.

If $(g', k, |G_o|)=(2, 0, g-1)$, then $\Theta$ is contained in one side of
$|\Sigma_g/G_o|$ in $|S^3/G_o|$. Hence the required $\mathbb{Z}_2\oplus\mathbb{Z}_2$-action does not exist.

If $(g', k, |G_o|)=(1, 2, g)$, then $|\Sigma_g/G_o|\simeq T^2$. In the proof of Proposition \ref{Pro of abelnot}, we only used homology theories, and the proof does not depend on $g>1$. Hence by a similar proof, the required $\mathbb{Z}_2\oplus\mathbb{Z}_2$-action on $(|\Sigma_g/G_o|,|S^3/G_o|)$ does not exist.

If $(g', k, |G_o|)=(0, 4, g+1)$, then $|\Sigma_g/G_o|\simeq S^2$ and $k=4$. Suppose $|\Sigma_g/G_o|$ intersects $\Theta$ at $\{A, B, C, D\}$, see Figure \ref{fig:twoBoneC}.

\begin{figure}[h]
\centerline{\scalebox{0.5}{\includegraphics{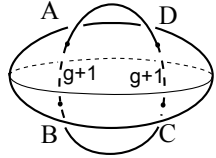}}}
\caption{$\Sigma_g/G_o$ and $\Theta$ in $S^3/G_o$}\label{fig:twoBoneC}
\end{figure}

If the required $\mathbb{Z}_2\oplus \mathbb{Z}_2$-action exists, then there exists $\eta\in \mathbb{Z}_2\oplus \mathbb{Z}_2$, it preserves the orientation of $S^3/G_o$ and reverses the orientation of $\Sigma_g/G_o$. Then on $\{A, B, C, D\}$ it has no fixed point. Otherwise, in $S^3/G$ the image of the fixed point will have non-abelian local group. Since $\eta$ changes the two sides of $S^3/G_o$, we can assume $\eta(AB)=AD$. Then $\eta$ will have order $4$, which is a contradiction.

Hence we have $(g', k, |G_o|)=(0, 6, g/2+1)$, and $|G|=2g+4$.

{\bf Case 2}: $G_o'\cong \mathbb{Z}_2\oplus \mathbb{Z}_2$ or
$|G_o'|=2$.

If $|G|=AE_g(Mix)>2g+4$, then $|G|>8$. Hence $|G_o'|\neq 2$. Then $G_o'\cong \mathbb{Z}_2\oplus \mathbb{Z}_2$, and $|G|=16$. By Proposition \ref{Pro of abelnot} and \ref{Pro of ORhandlebody}, $g$ can only be $4$.

Let $G'$ be the subgroup of $G$ containing all elements preserving the orientation of $\Sigma_4$. By the proof of Proposition \ref{Pro of ag}, the $G'$-action can extend to $V_4$. Since $|G/G_o'|\leq 4$ and $|G|=16$, we have $G/G_o'\cong \mathbb{Z}_2\oplus \mathbb{Z}_2$. Then $G'\cong \mathbb{Z}_2\oplus \mathbb{Z}_2\oplus \mathbb{Z}_2$ or $\mathbb{Z}_4\oplus \mathbb{Z}_2$. Let $\mathcal{G}$ be the induced finite graph of finite groups of $V_4/G'$. On one hand $\chi(\mathcal{G})=(1-4)/8=-3/8$. On the other hand, vertices and edges of $\mathcal{G}$ have orders $1$, $2$ or $4$. Hence $4\chi(\mathcal{G})\in \mathbb{Z}$. This is a contradiction.
\end{proof}

By Proposition \ref{Pro of cemp}--\ref{Pro of aeeven} and Remark \ref{Rem of classical result}, we finished the proof of Proposition \ref{Pro of maxorder of five type}, hence finished the proof of Theorem \ref{Thm of maximum of CandA action}.






\bibliographystyle{amsalpha}

\end{document}